\documentclass[11pt]{article}
\usepackage{graphicx}
\usepackage{geometry}
\usepackage{amsfonts,bbm}
\usepackage{amsthm}
\usepackage{amssymb}
\usepackage{amsmath,amsfonts}
\usepackage{mathtools, bm}
\usepackage{xcolor}
\usepackage{dsfont}
\usepackage{comment}
\usepackage{hyperref}
\usepackage[numbers]{natbib}
\usepackage{tikz}
\usetikzlibrary{arrows,decorations.pathreplacing, bending,hobby}
\usetikzlibrary{shapes.geometric}
\usetikzlibrary{decorations.markings}
\usepackage{subcaption}
\usepackage{cleveref}[capitalise, compress, nameinlink, noabbrev]

\theoremstyle{plain}
\newtheorem{theorem}{Theorem}
\crefname{theorem}{Theorem}{Theorems}

\newtheorem{proposition}[theorem]{Proposition}
\crefname{proposition}{Proposition}{Propositions}

\newtheorem{corollary}[theorem]{Corollary}
\crefname{corollary}{Corollary}{Corollaries}

\crefname{lemma}{Lemma}{Lemmata}

\crefname{conjecture}{Conjecture}{Conjectures}

\crefname{problem}{Problem}{Problems}

\newtheorem{claim}[theorem]{Claim}
\crefname{claim}{Claim}{Claims}

\crefname{observation}{Observation}{Observations}

\crefname{setup}{Setup}{Setups}

\crefname{fact}{Fact}{Facts}

\crefname{algorithm}{Algorithm}{Algorithms}

\crefname{remark}{Remark}{Remarks}

\crefname{example}{Example}{Examples}

\theoremstyle{definition}

\crefname{definition}{Definition}{Definitions}

\crefname{construction}{Construction}{Constructions}

\crefname{question}{Question}{Questions}

\numberwithin{equation}{section}

\crefname{section}{Section}{Sections}
\crefname{appendix}{Appendix}{Appendix}

\crefname{figure}{Figure}{Figures}

\newcommand\ceil[1]{\left\lceil#1\right\rceil}
\newcommand\floor[1]{\left\lfloor#1\right\rfloor}

\DeclareMathOperator{\AP}{AP}
\DeclareMathOperator{\MP}{MP}

\DeclareMathOperator{\ex}{ex}

\allowdisplaybreaks

\title{Ordered Ramsey and Turán numbers of alternating paths and their variants} 
\author{Gaurav Kucheriya\thanks{Department of Applied Mathematics, Charles University, Prague, Czechia. Email: \url{gaurav@kam.mff.cuni.cz}. Research supported by GAČR grant 25-17377S and GAUK project 92125.} \and Allan Lo\thanks{School of Mathematics, University of Birmingham, B15 2TT, UK. Email: \url{s.a.lo@bham.ac.uk}.} \and Jan Petr\thanks{Fakultät für Informatik und Mathematik, Universität Passau, Germany. Email: \url{jan.petr@uni-passau.de}, \url{amedeo.sgueglia@uni-passau.de}. JP is funded by the Deutsche Forschungsgemeinschaft (DFG, German Research Foundation) – 542321564. AS is funded by the Alexander von Humboldt Foundation.} \and Amedeo Sgueglia\footnotemark[3] \and Jun Yan\thanks{Mathematical Institute, University of Oxford, Oxford OX2 6GG, UK. Email: \url{jun.yan@maths.ox.ac.uk}. Supported by
ERC Advanced Grant 883810.}}

\begin{document}

\maketitle
\begin{abstract}
    An ordered graph is a graph whose vertex set is equipped with a total order. The ordered complete graph $K_N^<$ is the complete graph with vertex set $[N]$ equipped with the natural ordering of the integers.
    Given an ordered graph $H$, the ordered Ramsey number $R_<(H)$ is the smallest integer $N$ such that every red/blue edge-colouring of $K_N^<$ contains a monochromatic copy of $H$ with vertices appearing in the same relative order as in $H$.
    
    Balko, Cibulka, Kr\'al, and Kyn\v cl asked whether, among all ordered paths on $n$ vertices, the ordered Ramsey number is minimised by the alternating path $\AP_n$ -- the ordered path with vertex set $[n]$ such that the vertices encountered along the path are $1, n, 2, n - 1,3, n-2,\dots$\,. Motivated by this problem, we make progress on establishing the value of $R_<(\AP_n)$ by proving that 
    \[
    R_{<}(\AP_n)\leq \left(2+\frac{\sqrt{2}}{2}+o(1)\right)n.
    \]

    We then use similar methods to determine the exact ordered Tur\'an number of $\AP_n$, and study the ordered Ramsey and Tur\'an numbers of several related ordered paths.
\end{abstract}

\newpage
\section{Introduction}

An \emph{ordered graph} is a graph whose vertex set is equipped with a total order. 
We will equivalently view an $n$-vertex ordered graph as a graph whose vertices are labelled with $\{1,2,\dots,n\}$ (so the total order is the usual ordering of integers).
We say that an ordered graph $G$ \emph{contains} an ordered graph $H$ if there is an injection $f:V(H) \to V(G)$ that preserves both the vertex order and adjacency, that is, for all $i < j$ in $V(H)$, we have $f(i) < f(j)$ in $V(G)$, and $f(i)f(j) \in E(G)$ whenever $ij \in E(H)$. 
We say that $G$ \emph{avoids} $H$ if $G$ does not contain $H$. 
In recent years, there has been increasing interest in translating classical results on graphs to the ordered setting.
Here we consider ordered Ramsey and Tur\'an problems for different ordered paths.

\subsection{Ordered Ramsey number} 
The graph-theoretic version of Ramsey's theorem~\cite{ramsey1987problem} states that for every fixed graph $H$, there exists $N$ such that every red/blue edge-colouring of the complete graph on $N$ vertices contains a monochromatic copy of $H$. 
The smallest integer $N$ with the above property is called the \emph{Ramsey number} of $H$ and is denoted by $R(H)$.
The problem of quantitatively determining $R(H)$ has attracted a lot of attention over the past several decades, but has remained extremely challenging in general. 
Even when $H$ is the complete graph, little progress had been made since the upper bound of Erd\H{o}s and Szekeres~\cite{erdos1935combinatorial} in 1935 until the recent major breakthrough of Campos, Griffiths, Morris, and Sahasrabudhe~\cite{campos2023exponential}.

Recently, Balko, Cibulka, Král, and Kynčl~\cite{balko2013ramsey}, as well as Conlon, Fox, Lee, and Sudakov~\cite{conlon2017ordered}, independently initiated a systematic study of the analogue of Ramsey's theorem for ordered graphs. For a comprehensive survey of results in this area we refer to Balko~\cite{balko2025survey}.
Given an ordered graph $H$, the \emph{ordered Ramsey number} $R_<(H)$ is defined to be the smallest natural number $N$ such that every red/blue edge-colouring of $K_N^<$, the ordered complete graph on $\{1,2,\ldots, N\}$, contains a monochromatic ordered copy of $H$. 
It is straightforward to observe that $R_<(H)$ is well-defined as $R_<(H) \le R(K_{|H|})$.
Moreover, although it always holds that $R(H) \le R_<(H)$, the quantities $R(H)$ and $R_<(H)$ may behave very differently and a change in the labels of the vertices of $H$ may result in drastic changes in the ordered Ramsey number.

For example, this was observed for paths in~\cite{balko2013ramsey}.
Let $P_n$ be the $n$-vertex path on $[n]$.
In the unordered setting, Gerencsér and Gyárfás~\cite{gerencser1967ramsey} showed that $R(P_n)=\floor{3n/2}-1$, which is linear in $n$.
In the ordered setting, consider first the labelled path on $[n]$ where the vertices encountered along the path are $1, 2, \dots, n-1, n$.
This is usually known as the \emph{monotone ordered path} on $n$ vertices and is denoted by $\MP_n$.
The Erd\H{o}s–Szekeres theorem~\cite{erdos1935combinatorial}, dating back to the very origins of Ramsey theory, implies that $R_<(\MP_n) = (n - 1)^2 + 1$, which is quadratic in $n$.
On the other hand, Balko, Cibulka, Kr\'al, and Kyn\v cl~\cite[Proposition~3]{balko2013ramsey} showed that for a different vertex labelling, yielding the so-called alternating path, the ordered Ramsey number is again linear in $n$ and more closely resembles the linear behaviour of $R(P_n)$.
The interest in this particular ordered path is also due to a still open problem \cite[Problem 29]{balko2013ramsey} which asks whether among all the ordered paths on $n$ vertices, the alternating path minimises the ordered Ramsey number.

The \emph{alternating path} $\AP_n$ on $n$ vertices is the labelled path with vertex set $[n]$ such that the vertices encountered along the path are: 
$1,n,2,n-1,\ldots,n/2,n/2+1$, if $n$ is even; and $1,n,2,n-1,\ldots,\lfloor n/2 \rfloor ,\lfloor n/2\rfloor +2,\lfloor n/2\rfloor +1$, if $n$ is odd.
See~Figure~\ref{fig:orederpaths}(b) for a drawing of $AP_8$. It was proved in~\cite{balko2013ramsey} that
\[(5/2+o(1))n=5\lfloor n/2\rfloor-4\leq R_<(\AP_n)\leq 2n-3+\sqrt{2n^2-8n+11}=(2+\sqrt2+o(1))n.\] 
The proof of the upper bound above uses an extremal result on $\{0,1\}$-matrices established by Füredi and Hajnal~\cite[Theorem 2.2]{furedihajnaldavenport} and Cibulka~\cite[Corollary 1.9]{cibulka2013extremal}.
Later, Balko, Jel\'inek, and Valtr~\cite[Lemma 5.1]{balko2019ordered}, while studying the ordered Ramsey number of graphs with bounded maximum degree, used a different method and obtained the weaker upper bound $R_<(\AP_n)\leq(4+o(1))n$.
We refine their method and make progress towards determining the precise linear multiplicative factor.

\begin{theorem}\label{th:ramsey}
   $R_{<}(\AP_n)\leq2n-2+\floor{\frac{\sqrt{2(n-2)^2+(-1)^n}-1}{2}} = \left(2+\frac{\sqrt{2}}{2}+o(1)\right)n$.
\end{theorem}

In his master's thesis~\cite{poljak2023off}, Poljak listed the exact values or lower bounds on $R_<(\AP_n)$ up to $n=13$, which were obtained from computer experiments using a SAT solver algorithm.
It is surprising that, as shown in Table~\ref{table:value} below, the upper bound from Theorem~\ref{th:ramsey} matches all the known results reported in~\cite{poljak2023off}.
However, we do not believe that our bound is sharp for large~$n$. 

\begin{table}[h]
\centering
\begin{tabular}{||c|c|c|c|c|c|c|c|c|c|c|c|c||}
\hline
$n$      & $2$ & $3$ & $4$ & $5$ & $6$ & $7$ & $8$ & $9$ & $10$ & $11$ & $12$ & $13$ \\
\hline\hline
UB       & $2$ & $4$ & $7$ & $9$ & $12$ & $15$ & $17$ & $20$ & $23$ & $25$ & $28$ & $31$ \\
\hline
\cite{poljak2023off}  & $2$ & $4$ & $7$ & $9$ & $12$ & $15$ & $17$ & $20$ & $23$ & $\geq25$ & $\geq28$ & $\geq31$ \\
\hline
\end{tabular}
\caption{Upper bounds (UB) from Theorem~\ref{th:ramsey} and results from~\cite{poljak2023off} on $R_<(\AP_n)$.}\label{table:value}
\end{table}

\subsection{Ordered Tur\'an numbers}
Tur\'an-type problems constitute another major research direction in extremal graph theory and are closely related to Ramsey-type problems. 
Given a graph $H$, the goal is to determine the \emph{Tur\'an number} of $H$, denoted by $\ex(n,H)$, which is the maximum number of edges an $n$-vertex graph can have if it does not contain $H$.
The famous Erd\H{o}s--Stone--Simonovits theorem~\cite{ErdosSimonovits1966,ErdosStone1946} determines $\ex(n,H)$ asymptotically for every non-bipartite graph $H$, in terms of its chromatic number.
In 2006, Pach and Tardos~\cite{pach2006forbidden} introduced Tur\'an-type problems for ordered graphs.
Analogous to the setting of (unordered) graphs, for an ordered graph~$H$, the object of interest is the \emph{ordered Tur\'an number} $\ex_<(n,H)$, defined as the maximum number of edges an ordered graph with vertex set $\{1,2,\ldots,n\}$ can have if it avoids $H$.
Pach and Tardos obtained a generalization of the Erd\H{o}s--Stone--Simonovits theorem in this setting, with the role of chromatic number replaced by a related parameter which they called the interval chromatic number. 

We determine the exact ordered Tur\'an number of $\AP_n$. 
\begin{theorem}\label{th:turan}
Let $N\geq n$. Then $\ex_<(N,\AP_n)=\binom{n-1}2+(n-2)(N-n+1)$.
\end{theorem}

We remark that one could deduce an upper bound on $R_<(\AP_n)$ by considering the majority colour in a red/blue edge-colouring of the ordered complete graph, and applying Theorem~\ref{th:turan} to the subgraph induced by the edges of that colour.
However, this would only recover the weaker bound $R_<(\AP_n)\leq(2+\sqrt{2}+o(1))n$ from~\cite{balko2013ramsey}.

\subsection{Other ordered paths}
\label{sec1.3}
Consider for simplicity the case when $n$ is even, i.e., $n=2k$ for some $k \in \mathbb{N}$. 
Observe that the alternating path $\AP_n$ has the following equivalent description:
As we traverse the path starting from the endpoint in $A=\{1,2,\ldots,k\}$, we alternate between vertices of $A$ and $B=\{k+1,k+2,\ldots,2k\}$, and encounter the vertices of $A$ in increasing order and the vertices of $B$ in decreasing order.
If we still alternate between $A$ and $B$, but reverse the order of appearance of vertices in one or both of the intervals $A$ and $B$, then we obtain three other ordered paths on $\{1,2,\ldots,n\}$. We denote them as $P_n^{r_1,r_2}$, where $r_1$ (resp. $r_2$) is $<$ if, starting from the endpoint in $A$, the vertices of $A$ (resp. $B$) are encountered in increasing order, and $>$ if they are in decreasing order. 
With this notation, we have $\AP_n=P_n^{<,>}$.
Figure~\ref{fig:orederpaths} shows the four possible ordered paths on 8 vertices that can arise this way.
\begin{figure}[h]
\centering
\begin{subfigure}[c]{0.48\textwidth}
\centering
\scalebox{1}{%
\begin{tikzpicture}
\fill (0,0) circle[radius=2pt];
\fill (0.5,0) circle[radius=2pt];
\fill (1,0) circle[radius=2pt];
\fill (1.5,0) circle[radius=2pt];
\fill (2,0) circle[radius=2pt];
\fill (2.5,0) circle[radius=2pt];
\fill (3,0) circle[radius=2pt];
\fill (3.5,0) circle[radius=2pt];

\begin{scope}
    \clip (0,0) rectangle (3.5,2.5);
    \draw (1.0,0) circle(1.0);
    \draw (1.5,0) circle(1.0);
    \draw (2.0,0) circle(1.0);
    \draw (2.5,0) circle(1.0);
    \draw (1.25,0) circle(0.75);
    \draw (1.75,0) circle(0.75);
    \draw (2.25,0) circle(0.75);

\end{scope}

\node[align=center] at (0,-0.25) {$1$};
\node[align=center] at (0.5,-0.25) {$2$};
\node[align=center] at (1,-0.25) {$3$};
\node[align=center] at (1.5,-0.25) {$4$};
\node[align=center] at (2,-0.25) {$5$};
\node[align=center] at (2.5,-0.25) {$6$};
\node[align=center] at (3,-0.25) {$7$};
\node[align=center] at (3.5,-0.25) {$8$};
\end{tikzpicture}
}
\caption{$P^{<,<}_8$}
\end{subfigure}
\begin{subfigure}[c]{0.48\textwidth}
\centering
\scalebox{1.0}{%
\begin{tikzpicture}
\fill (0,0) circle[radius=2pt];
\fill (0.5,0) circle[radius=2pt];
\fill (1,0) circle[radius=2pt];
\fill (1.5,0) circle[radius=2pt];
\fill (2,0) circle[radius=2pt];
\fill (2.5,0) circle[radius=2pt];
\fill (3,0) circle[radius=2pt];
\fill (3.5,0) circle[radius=2pt];

\begin{scope}
    \clip (0,0) rectangle (3.5,2.5);
    \draw (2.0,0) circle(0.5);
    \draw (2.0,0) circle(1.0);
    \draw (2.0,0) circle(1.5);
    \draw (1.75,0) circle(0.25);
    \draw (1.75,0) circle(0.75);
    \draw (1.75,0) circle(1.25);
    \draw (1.75,0) circle(1.75);

\end{scope}

\node[align=center] at (0,-0.25) {$1$};
\node[align=center] at (0.5,-0.25) {$2$};
\node[align=center] at (1,-0.25) {$3$};
\node[align=center] at (1.5,-0.25) {$4$};
\node[align=center] at (2,-0.25) {$5$};
\node[align=center] at (2.5,-0.25) {$6$};
\node[align=center] at (3,-0.25) {$7$};
\node[align=center] at (3.5,-0.25) {$8$};
\end{tikzpicture}
}
\caption{$P^{<,>}_8=AP_8$}
\end{subfigure}
\begin{subfigure}[c]{0.48\textwidth}
\centering
\scalebox{1}{%
\begin{tikzpicture}
\fill (0,0) circle[radius=2pt];
\fill (0.5,0) circle[radius=2pt];
\fill (1,0) circle[radius=2pt];
\fill (1.5,0) circle[radius=2pt];
\fill (2,0) circle[radius=2pt];
\fill (2.5,0) circle[radius=2pt];
\fill (3,0) circle[radius=2pt];
\fill (3.5,0) circle[radius=2pt];

\begin{scope}
    \clip (0,0) rectangle (3.5,2.5);
    \draw (1.5,0) circle(0.5);
    \draw (1.5,0) circle(1.0);
    \draw (1.5,0) circle(1.5);
    \draw (1.75,0) circle(0.25);
    \draw (1.75,0) circle(0.75);
    \draw (1.75,0) circle(1.25);
    \draw (1.75,0) circle(1.75);

\end{scope}

\node[align=center] at (0,-0.25) {$1$};
\node[align=center] at (0.5,-0.25) {$2$};
\node[align=center] at (1,-0.25) {$3$};
\node[align=center] at (1.5,-0.25) {$4$};
\node[align=center] at (2,-0.25) {$5$};
\node[align=center] at (2.5,-0.25) {$6$};
\node[align=center] at (3,-0.25) {$7$};
\node[align=center] at (3.5,-0.25) {$8$};
\end{tikzpicture}
}
\caption{$P^{>,<}_8$}
\end{subfigure}
\begin{subfigure}[c]{0.48\textwidth}
\centering
\scalebox{1.0}{%
\begin{tikzpicture}
\fill (0,0) circle[radius=2pt];
\fill (0.5,0) circle[radius=2pt];
\fill (1,0) circle[radius=2pt];
\fill (1.5,0) circle[radius=2pt];
\fill (2,0) circle[radius=2pt];
\fill (2.5,0) circle[radius=2pt];
\fill (3,0) circle[radius=2pt];
\fill (3.5,0) circle[radius=2pt];

\begin{scope}
    \clip (0,0) rectangle (3.5,2.5);
    \draw (1.0,0) circle(1.0);
    \draw (1.5,0) circle(1.0);
    \draw (2.0,0) circle(1.0);
    \draw (2.5,0) circle(1.0);
    \draw (1.25,0) circle(1.25);
    \draw (1.75,0) circle(1.25);
    \draw (2.25,0) circle(1.25);

\end{scope}

\node[align=center] at (0,-0.25) {$1$};
\node[align=center] at (0.5,-0.25) {$2$};
\node[align=center] at (1,-0.25) {$3$};
\node[align=center] at (1.5,-0.25) {$4$};
\node[align=center] at (2,-0.25) {$5$};
\node[align=center] at (2.5,-0.25) {$6$};
\node[align=center] at (3,-0.25) {$7$};
\node[align=center] at (3.5,-0.25) {$8$};
\end{tikzpicture}
}
\caption{$P^{>,>}_8$}
\end{subfigure}
\caption{The four ordered paths on $8$ vertices of the form $P_8^{r_1,r_2}$.}\label{fig:orederpaths}
\end{figure}
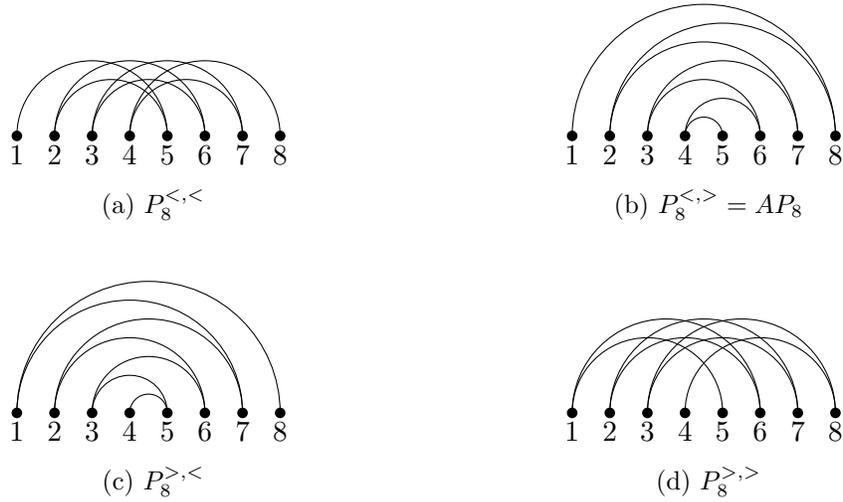

Note that $P_n^{>,<}$ is isomorphic to $\AP_n$ by reversing the order of the vertices, so Theorems~\ref{th:ramsey} and~\ref{th:turan} directly extend to this path.
\begin{corollary}
Let $n$ be even. Then $R_{<}(P_n^{>,<})\leq2n-2+\floor{\frac{\sqrt{2(n-2)^2+(-1)^n}-1}{2}}$.
\end{corollary}    

\begin{corollary}
Let $n$ be even and $N \geq n$. Then $\mathrm{ex}_<(N,P_n^{>,<})=\binom{n-1}{2}+(n-2)(N-n+1)$.
\end{corollary}

For the remaining two ordered paths, we can prove the following upper bounds on their ordered Ramsey and Turán numbers. The results are not sharp, and obtaining the precise values would be an interesting future research direction.
\begin{proposition}\label{prop:ramsey}
Let $n$ be even and $P \in \{P_n^{<,<},P_n^{>,>}\}$. Then $R_<(P)\leq 3n-4$.
\end{proposition}

\begin{theorem}\label{thm:turanother}
Let $n$ and $N$ be even, $N \geq n$ and $P \in \{P_n^{<,<},P_n^{>,>}\}$. Then, $\mathrm{ex}_<(N,P)\leq nN(\log_2(N/n)+2)$.
\end{theorem}

We conclude by noting that there are examples demonstrating that the ordered Tur\'an numbers of $P^{<,<}_n$ and $P^{>,>}_n$ can differ both with each other and from that of $P^{<,>}_n$.
Moreover, we can determine the exact ordered bipartite Tur\'an number for all the four ordered paths above.
We refer the reader to Section~\ref{sec4} for more details.

\bigskip
\noindent {\bf Remark.} While working on this project, we learnt of independent work by Bar\'at, Freschi, and T\'oth~{\cite{barat2025matchings}} who, using similar methods, proved the weaker bound $R_<(\AP_n)\leq 3n+3$ and obtained the same Tur\'an result of Theorem~\ref{th:turan}.

\bigskip
\noindent \textbf{Notation.}
For any positive integer $n$, we use $[n]$ to denote the set of integers $\{1,2,\ldots,n\}$. For any two integers $a\leq b$, we use $[a,b]$ to denote the set $\{a,a+1,\ldots,b\}$ of all integers from $a$ to $b$. We use $K_n^<$ to denote the ordered complete graph on $[n]$ with the usual ordering.

Following previous work (see~e.g.~\cite{balko2013ramsey,neidinger2019ramsey,poljak2023off}), we often visualise a red/blue edge-colouring of $K_N^<$ through an $N \times N$ upper triangular matrix obtained as follows: Given a red/blue edge-coloured~$K_N^<$, for $1 \le i < j \le N$, we colour the entry~$(i,j)$ with the colour of the edge~$ij$.
Similarly, for an ordered graph~$H$ whose vertex set is a subset of $[N]$, we can represent it as follows: For every $1\leq i<j\leq N$, we colour the entry $(i,j)$ green if and only if $ij$ is an edge in $H$.
As an example, for an even number $n<N$, a copy of the alternating path $\AP_n$ in $K_N^<$ would correspond to the $n-1$ entries

\[(i_1, i_n), (i_2, i_n),(i_2,i_{n-1}),\ldots,(i_{n/2}, i_{n/2+2}),(i_{n/2}, i_{n/2+1})\]
with $i_1 < i_2 < \dots < i_n$. 
In particular, $\AP_n$ resembles a ``staircase'' in the matrix representation.
Several illustrative examples are shown in Figure~\ref{fig:matrices}.

\begin{figure}[h]
\centering
\begin{subfigure}[c]{0.48\textwidth}
\centering
\scalebox{1.25}{%
\begin{tikzpicture}
\fill (0,0) circle[radius=2pt];
\fill (0.5,0) circle[radius=2pt];
\fill (1,0) circle[radius=2pt];
\fill (1.5,0) circle[radius=2pt];
\fill (2,0) circle[radius=2pt];
\fill (2.5,0) circle[radius=2pt];
\fill (3,0) circle[radius=2pt];

\begin{scope}
    \clip (0,0) rectangle (3,2);
    \draw (1.5,0) circle(1.5);
    \draw (1.75,0) circle(1.25);
    \draw (1.5,0) circle(1);
    \draw (1.75,0) circle(0.75);
    \draw (1.5,0) circle(0.5);
    \draw (1.75,0) circle(0.25);
\end{scope}

\node[align=center] at (0,-0.25) {$1$};
\node[align=center] at (0.5,-0.25) {$2$};
\node[align=center] at (1,-0.25) {$3$};
\node[align=center] at (1.5,-0.25) {$4$};
\node[align=center] at (2,-0.25) {$5$};
\node[align=center] at (2.5,-0.25) {$6$};
\node[align=center] at (3,-0.25) {$7$};
\end{tikzpicture}
}
\end{subfigure}
\begin{subfigure}[c]{0.48\textwidth}
\centering
\scalebox{1.0}{%
\begin{tikzpicture}

\fill[green,opacity=0.7] (3.0,1.5) -- (3.0,2.5) -- (3.5,2.5) -- (3.5,1.5);
\fill[green,opacity=0.7] (3.5,2.0) -- (3.5,3.0) -- (4.0,3.0) -- (4.0,2.0);
\fill[green,opacity=0.7] (4.0,3.5) -- (4.0,2.5) -- (4.5,2.5) -- (4.5,3.5);

\node[align=center] at (1.25,3.75) {$1$};
\node[align=center] at (1.75,3.75) {$2$};
\node[align=center] at (2.25,3.75) {$3$};
\node[align=center] at (2.75,3.75) {$4$};
\node[align=center] at (3.25,3.75) {$5$};
\node[align=center] at (3.75,3.75) {$6$};
\node[align=center] at (4.25,3.75) {$7$};

\node[align=center] at (0.75,3.25) {$1$};
\node[align=center] at (0.75,2.75) {$2$};
\node[align=center] at (0.75,2.25) {$3$};
\node[align=center] at (0.75,1.75) {$4$};
\node[align=center] at (0.75,1.25) {$5$};
\node[align=center] at (0.75,0.75) {$6$};
\node[align=center] at (0.75,0.25) {$7$};


\draw (4.0,0.5) -- (4.5,0.5);
\draw (3.5,1.0) -- (4.5,1.0);
\draw (3.0,1.5) -- (4.5,1.5);
\draw (2.5,2.0) -- (4.5,2.0);
\draw (2.0,2.5) -- (4.5,2.5);
\draw (1.5,3.0) -- (4.5,3.0);
\draw (1.5,3.5) -- (4.5,3.5);


\draw (4.5,0.5) -- (4.5,3.5);
\draw (4.0,0.5) -- (4.0,3.5);
\draw (3.5,1.0) -- (3.5,3.5);
\draw (3.0,1.5) -- (3.0,3.5);
\draw (2.5,2.0) -- (2.5,3.5);
\draw (2.0,2.5) -- (2.0,3.5);
\draw (1.5,3.0) -- (1.5,3.5);
\end{tikzpicture}
}
\end{subfigure}

\vspace{10pt}

\centering
\begin{subfigure}[c]{0.48\textwidth}
\centering
\scalebox{1.25}{%
\begin{tikzpicture}
\begin{scope}
    \clip (0,0) rectangle (2.5,2);
    
    \draw[blue] (0.25,0) circle(0.25);
    \draw[red] (0.75,0) circle(0.25);
    \draw[green, line width=3, opacity=0.6]
    (1.25,0) circle(0.25);
    \draw[blue] (1.25,0) circle(0.25);
    \draw[blue] (1.75,0) circle(0.25);
    \draw[red] (2.25,0) circle(0.25);
    
    \draw[red] (0.5,0) circle(0.5);
    \draw[blue] (1.0,0) circle(0.5);
    \draw[red] (1.5,0) circle(0.5);
    \draw[red] (2.0,0) circle(0.5);

    \draw[blue] (0.75,0) circle(0.75);
    \draw[red] (1.25,0) circle(0.75);
    \draw[green, line width=3, opacity=0.6]
    (1.75,0) circle(0.75);
    \draw[blue] (1.75,0) circle(0.75);

    \draw[blue] (1.0,0) circle(1.0);
    \draw[green, line width=3, opacity=0.6]
    (1.5,0) circle(1.0);
    \draw[blue] (1.5,0) circle(1.0);

    \draw[red] (1.25,0) circle(1.25);
\end{scope}

\fill (0,0) circle[radius=2pt];
\fill (0.5,0) circle[radius=2pt];
\fill (1,0) circle[radius=2pt];
\fill (1.5,0) circle[radius=2pt];
\fill (2,0) circle[radius=2pt];
\fill (2.5,0) circle[radius=2pt];

\node[align=center] at (0,-0.25) {$1$};
\node[align=center] at (0.5,-0.25) {$2$};
\node[align=center] at (1,-0.25) {$3$};
\node[align=center] at (1.5,-0.25) {$4$};
\node[align=center] at (2,-0.25) {$5$};
\node[align=center] at (2.5,-0.25) {$6$};
\end{tikzpicture}
}

\end{subfigure}
\begin{subfigure}[c]{0.48\textwidth}
\centering
\scalebox{1.0}{%
\begin{tikzpicture}

\fill[blue,opacity=0.7] (2.0,3.0) -- (2.0,3.5) -- (1.5,3.5) -- (1.5,3.0);
\fill[red,opacity=0.7] (2.0,3.0) -- (2.0,3.5) -- (2.5,3.5) -- (2.5,3.0);
\fill[blue,opacity=0.7] (3.0,3.0) -- (3.0,3.5) -- (2.5,3.5) -- (2.5,3.0);
\fill[blue,opacity=0.7] (3.0,3.0) -- (3.0,3.5) -- (3.5,3.5) -- (3.5,3.0);
\fill[red,opacity=0.7] (4.0,3.0) -- (4.0,3.5) -- (3.5,3.5) -- (3.5,3.0);

\fill[red,opacity=0.7] (2.0,3.0) -- (2.0,2.5) -- (2.5,2.5) -- (2.5,3.0);
\fill[blue,opacity=0.7] (3.0,3.0) -- (3.0,2.5) -- (2.5,2.5) -- (2.5,3.0);
\fill[red,opacity=0.7] (3.0,3.0) -- (3.0,2.5) -- (3.5,2.5) -- (3.5,3.0);
\filldraw[fill=white, draw=green, opacity=0.7, line width=1mm] (3.55, 2.55) rectangle +(0.4, 0.4);
\fill[blue,opacity=0.7] (3.6,2.9) -- (3.6,2.6) -- (3.9,2.6) -- (3.9,2.9);

\filldraw[fill=white, draw=green, opacity=0.7, line width=1mm] (2.55, 2.05) rectangle +(0.4, 0.4);
\fill[blue,opacity=0.7] (2.9,2.1) -- (2.9,2.4) -- (2.6,2.4) -- (2.6,2.1);
\fill[red,opacity=0.7] (3.0,2.0) -- (3.0,2.5) -- (3.5,2.5) -- (3.5,2.0);
\filldraw[fill=white, draw=green, opacity=0.7, line width=1mm] (3.55, 2.05) rectangle +(0.4, 0.4);
\fill[blue,opacity=0.7] (3.6,2.1) -- (3.6,2.4) -- (3.9,2.4) -- (3.9,2.1);

\fill[blue,opacity=0.7] (3.0,2.0) -- (3.0,1.5) -- (3.5,1.5) -- (3.5,2.0);
\fill[red,opacity=0.7] (4.0,2.0) -- (4.0,1.5) -- (3.5,1.5) -- (3.5,2.0);

\fill[red,opacity=0.7] (4.0,1.0) -- (4.0,1.5) -- (3.5,1.5) -- (3.5,1.0);

\node[align=center] at (1.25,3.75) {$1$};
\node[align=center] at (1.75,3.75) {$2$};
\node[align=center] at (2.25,3.75) {$3$};
\node[align=center] at (2.75,3.75) {$4$};
\node[align=center] at (3.25,3.75) {$5$};
\node[align=center] at (3.75,3.75) {$6$};

\node[align=center] at (0.75,3.25) {$1$};
\node[align=center] at (0.75,2.75) {$2$};
\node[align=center] at (0.75,2.25) {$3$};
\node[align=center] at (0.75,1.75) {$4$};
\node[align=center] at (0.75,1.25) {$5$};
\node[align=center] at (0.75,0.75) {$6$};

\draw (3.5,1.0) -- (4.0,1.0);
\draw (3.0,1.5) -- (4.0,1.5);
\draw (2.5,2.0) -- (4.0,2.0);
\draw (2.0,2.5) -- (4.0,2.5);
\draw (1.5,3.0) -- (4.0,3.0);
\draw (1.5,3.5) -- (4.0,3.5);

\draw (4.0,1.0) -- (4.0,3.5);
\draw (3.5,1.0) -- (3.5,3.5);
\draw (3.0,1.5) -- (3.0,3.5);
\draw (2.5,2.0) -- (2.5,3.5);
\draw (2.0,2.5) -- (2.0,3.5);
\draw (1.5,3.0) -- (1.5,3.5);
\end{tikzpicture}
}
\end{subfigure}
\caption{Above: The ordered path $\AP_7$ and its matrix representation. Below: A red/blue edge-colouring of $K_6^<$ and its matrix representation. A monochromatic subgraph isomorphic to $\AP_4$ (in fact, the unique one) is highlighted in green.}\label{fig:matrices}
\end{figure}
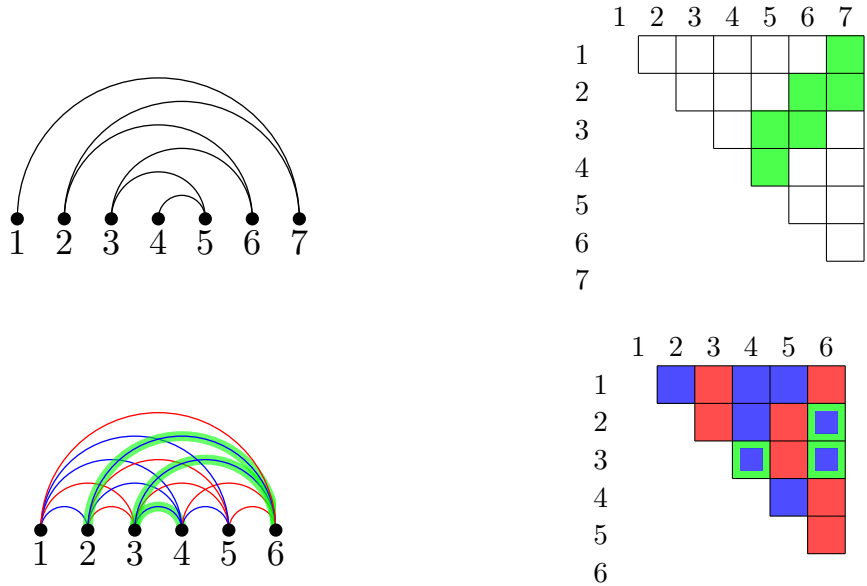

\medskip
\noindent \textbf{Organization.}
The rest of the paper is organized as follows. We prove Theorem~\ref{th:ramsey} in Section~\ref{sec2}, Theorem~\ref{th:turan} in Section~\ref{sec3}, and Proposition~\ref{prop:ramsey} and Theorem~\ref{thm:turanother} in Section~\ref{sec4}.

\section{Proof of Theorem~\ref{th:ramsey}}\label{sec2}

In this section we prove Theorem~\ref{th:ramsey} using a refinement of the arguments of Balko, Jel\'inek, and Valtr in~\cite[Lemma 5.1]{balko2019ordered}.
They showed that any ordered graph $G$ on $N \geq n/ \varepsilon$ vertices and at least $\varepsilon N^2$ edges contains a copy of $\AP_n$.
To this end, they designed an edge-deletion algorithm with the following property: If any edge of $G$ remains at the end, then $G$ contains a copy of $\AP_n$.
The result then follows by bounding the total number of edges removed by the algorithm.

Similarly, here we want to show that, for $N$ as in Theorem~\ref{th:ramsey}, every red/blue edge-colouring of $K_N^<$ contains a monochromatic copy of $\AP_n$.
The improvement in our approach comes from the following idea.
Instead of looking for any monochromatic $\AP_n$, we will restrict our attention to those where vertices labelled $1, 2, \dots, \lceil n/2\rceil$ lie in a set $A \subseteq [N]$ and the vertices labelled $\lceil n/2 \rceil+1, \lceil n/2 \rceil+2,\dots,n$ belong to a set $B \subseteq [N]$, with $A$ and $B$ not necessarily disjoint.
Then, we identify and count some edges, which we will call grey edges, that cannot be present in any such copy of $\AP_n$.
After removing these grey edges we apply a similar edge-deletion algorithm on the remaining edges between $A$ and $B$.
An example application of the algorithm we use is illustrated in Figure~\ref{fig:thm2}.

\begin{proof}[Proof of Theorem~\ref{th:ramsey}]
If $n$ is even, set $N=2n-2+\floor{\frac{\sqrt{2n^2-8n+9}-1}{2}}$ and ${{a}}=N - n+1$.
If $n$ is odd, set $N=2n-2+\floor{\frac{\sqrt{2n^2-8n+7}-1}{2}}$ and ${{a}}=N-n+2$.
In both cases, set $A=[1,{{a}}]$ and $B=[n,N]$, and fix a red/blue edge-colouring of $K_N^<$. 

To compute the number $e(A,B)$ of edges between $A$ and $B$, we subtract from $A \times B$ all pairs of the form $(x,x)$, which are not edges, and account for pairs $(x,y)$ such that $(y,x)$ is also in $A \times B$, since these are double-counted.
Hence,
\begin{align*}
e(A,B)&= |A||B|-|A\cap B|-\binom{|A\cap B|}{2}\\
&= {{a}} (N-n+1) - ({{a}}-n+1)-\binom{{{a}}-n+1}{2}\\
&= \begin{cases} \frac12 {N(N-1)}- (n-1)(n-2) & \text{if $n$ is even,}\\
\frac12 N(N-1) - (n-2)^2& \text{if $n$ is odd.}\end{cases}
\end{align*}

We will run an algorithm to search for a monochromatic copy of $\AP_n$ whose first $\ceil{n/2}$ vertices belong to $A$ and the last $\floor{n/2}$ vertices lie in $B$. 

First, we identify a set $F$ of edges in $E(A,B)$ that cannot be present in any (monochromatic or not) such copy of $\AP_n$. We will call them \emph{grey edges}.
For example, suppose that $1$ is a vertex in such an $\AP_n$, which we denote by $P$.
Then $1$ has to be the leftmost endpoint of $P$ and its neighbour in $P$, say $w$, must belong to $B$ and be the rightmost vertex among those in $P$.
The path $P$ has $n-2$ vertices between $1 $ and $w$, $\floor{n/2}-1$ of which belong to~$B$.
This implies that $w$ cannot be any of the first $\floor{n/2}-1$ vertices of $B$, and thus all the edges between $1 \in A$ and the first $\floor{n/2}-1$ vertices of $B$ are grey edges.
Similarly, the edges between every vertex~$i \in [ 1,\floor{n/2}-1] \subseteq  A$ and the first $\floor{n/2}-i$ vertices of $B$ are grey edges, and the edges between every vertex $N-i+1 \in [N-\ceil{n/2}+3, N] \subseteq B$ and the last $\ceil{n/2}-i-1$ vertices of $A$ are grey edges. Let $F$ be this set of edges.

Observe that the edges mentioned above are pairwise distinct, so the total number of grey edges in $F$ is
\[f:=|F|=\binom{\floor{n/2}}{2}+\binom{\ceil{n/2}-1}{2}=\begin{cases}\frac14(n-2)^2 & \text{if $n$ is even,}\\\frac14(n-1)(n-3) & \text{if $n$ is odd.}\end{cases}\]

\input{thm2}

Now, we consider the following algorithm which removes edges in $E(A,B)\setminus F$ in $n-2$ steps as follows. 
In every odd step $2i-1$ with $i\in[1,\ceil{n/2}-1]$, we remove, from each vertex in $$I_i:=[\floor{3n/2}-i,N-i+1] \subseteq B,$$ the leftmost remaining red edge and the leftmost remaining blue edge to $A$ (if they exist).
In every even step $2i$ with $i\in[1,\floor{n/2}-1]$, we remove, from each vertex in $$J_i:=[i+1,{{a}}-\ceil{n/2}+i+1] \subseteq A,$$ the rightmost remaining red edge and the rightmost remaining blue edge to $B$ (if they exist).

\begin{claim}
\label{claim:algorithm}
    Let $k \in [1,n-2]$.
    If $k=2i-1$ is odd, then after step $k$, every edge in $E(A,B)$ between $[1,i]$ and $B$ is either a grey edge or has been removed.
    If $k=2i$ is even, then after step $k$, every edge in $E(A,B)$ between $A$ and $[N-i+1,N]$ is either a grey edge or has been removed.
\end{claim}

\begin{proof}
    We prove the statement by induction on $k$.
    Suppose $k=1$, then $i=1$.
    Recall from above that all the edges between $1$ and $[n,\lfloor 3n/2 \rfloor-2]$ are grey edges.
    Moreover, in step 1, the algorithm removes the leftmost red edge and the leftmost blue edge from every vertex in $[ \lfloor 3n/2 \rfloor-1, N]$ to $A$, so every edge between $[\lfloor 3n/2 \rfloor-1, N]$ and $1$ (regardless of its colour) is removed in step 1.
    Since $B=[n,\lfloor 3n/2 \rfloor-2] \cup [\lfloor 3n/2 \rfloor-1, N]$, the statement holds for $k=1$. The case when $k=2$ can be proved similarly. 

    Now suppose the statement holds up to $k-1$, we will prove the case when $k=2i$ is even. The case when $k$ is odd is analogous.
    By induction, after step $k-1$, all the edges in $E(A,B)$ between $[1,i]$ and $B$, and between $A$ and $[N-i+2,N]$ have been removed or are grey edges.
    In order to prove the statement, it suffices to show that after step $k$, every edge in $E(A,B)$ between $A$ and $N-i+1$ has been removed or is a grey edge.
    Recall that the edges between $[{{a}}-(\lceil n/2\rceil -i-2),{{a}}]$ and $N-i+1$ are grey edges.
    By assumption, the edges between $[1,i]$ and $B$ have already been removed or are grey edges, so in particular, this is true for the edges between $[1,i]$ and $N-i+1$. 
    Moreover, the algorithm removes the rightmost remaining red edge and the rightmost remaining blue edge from every vertex in $J_i$ to $B$.
    However, every vertex in $J_i\subseteq A$ is no longer connected to any vertex in $[N-i+2,N]$, so if it is still connected to $N-i+1$, then the corresponding edge (regardless of its colour) gets removed in step $k$ of the algorithm.
    Since $A=[1,i] \cup J_i \cup [{{a}}-(\lceil n/2\rceil -i-2),{{a}}]$, the statement holds for $k$ as well, finishing the induction.
\renewcommand{\qedsymbol}{$\boxdot$}
\end{proof}
\renewcommand{\qedsymbol}{$\square$}

\begin{claim}
\label{claim:no_edges}
    Suppose that an edge in $E(A,B) \setminus F$ remains after step $n-2$ of the algorithm.
    Then $K_N^<$ contains a monochromatic copy of $\AP_n$.
\end{claim}

\begin{proof}
    We prove the claim when $n$ is even.
    The case when $n$ is odd is similar.

    Let $v_{n-1}v_{n}$ be an edge of $E(A,B) \setminus F$ which remains after step $n-2$ of the algorithm.
    Without loss of generality, assume that $v_{n-1}v_n$ is coloured blue, and $v_{n-1} \le v_n$, so $v_{n-1} \in A$ and $v_n \in B$. 
    By Claim~\ref{claim:algorithm}, after step $n-2$, all the edges (of any colour) between $[1,n/2-1]$ and $B$, and between $A$ and $[N-n/2+2,N]$ have been removed or are grey edges.
    Therefore, $v_{n-1} \ge n/2$ and $v_n \le N-n/2+1$.
    
    We will show that we can backtrack the algorithm and build a blue copy of $\AP_n$. We do this by showing with a backward induction that for $k=n-1,n-2,\ldots,1$, there exists a vertex $v_k$ such that $v_k v_{k+1} \dots v_{n-1} v_n$ is an alternating blue path and all its edges remain after step $k-1$ of the algorithm.
    When $k=n-1$, the statement is true by assumption.
    Now suppose it holds for some $k\geq2$, we need to prove that it also holds for $k-1$.
    Assume that $k=2i$ is even, the case when it is odd is similar. 
    By induction hypothesis, $v_{k} v_{k+1} \dots v_n$ is a blue alternating path and all of its edges remain after step $k-1=2i-1$ of the algorithm. In particular, as $v_kv_{k+1}$ remains after step $2i-2$, $v_{k+1}\in A$, $v_k\in B$, and we have $v_k\leq N-i+1$ by Claim~\ref{claim:algorithm}. Also, as there are $n/2-i$ vertices among $v_{k}, v_{k+1}, \ldots, v_n$ on the right of $v_n$, and $v_n\in B$, we must have $v_{k} \ge v_n + n/2-i\geq 3n/2-i$. 
    
    Therefore, $v_k \in I_i$. Since $v_{k+1}v_k$ is a blue edge which remains after step $k-1=2i-1$ and $v_k\in I_i$, the algorithm must have removed a blue edge of the form $w v_k$ for some vertex~$w$ on the left of $v_{k+1}$ in step $k-1$.
    This implies that the edge $wv_k$ remains after step $k-2$ of the algorithm, so letting $v_{k-1}=w$ extends the alternating path found so far, and completes the induction step.
\renewcommand{\qedsymbol}{$\boxdot$}
\end{proof}
\renewcommand{\qedsymbol}{$\square$}

Let $r$ be the number of edges in $E(A,B) \setminus F$ removed by the algorithm during the $n-2$ steps.
Then, 
\begin{align*}
r& \le 2 \cdot \sum_{i=1}^{\ceil{n/2}-1} |I_i| + 2 \cdot \sum_{i=1}^{\floor{n/2}-1} |J_i| \\
&=2(\ceil{n/2}-1)(N-\floor{3n/2}+2)+ 2(\floor{n/2}-1)({{a}}-\ceil{n/2}+1)\\
&\phantom{:}=\begin{cases}
  (n-2) (N + {{a}} - 2n+3)  = (n-2)(2N-3n+4)& \text{if $n$ is even,}\\
2(n-2)(N-n+1) - (n-3)(N -{{a}}) = (n-2)(2N-3n+5) & \text{if $n$ is odd.}
\end{cases}
\end{align*}

Observe that $ r+f < e(A,B) $ by our choices of $N$.
Thus an edge in $E(A,B) \setminus F$ remains after step $n-2$ of the algorithm, so by Claim~\ref{claim:no_edges}, we have $R_<(\AP_n) \le N$.
\end{proof}

\section{Proof of Theorem~\ref{th:turan}}\label{sec3}

The upper bound on $\ex_<(N,\AP_n)$ is proven in a fashion similar to the proof of Theorem~\ref{th:ramsey}. 
We run the same edge-deletion algorithm (see the short description at the beginning of Section~\ref{sec2}), but this time we take $A=B=[N]$.
Again we identify edges that cannot be in any $\AP_n$, so that it suffices to run the algorithm on the remaining edges.
The outcome of the described algorithm for one specific ordered graph as well as an extremal construction are illustrated in Figure~\ref{fig:thm3}.

\begin{proof}[Proof of Theorem~\ref{th:turan}]
First, we prove that $\ex_<(N,\AP_n)\leq \binom{n-1}2+(n-2)(N-n+1)$. Let $G$ be any ordered graph on $[N]$.
We show that if $e(G) \geq \binom{n-1}2+(n-2)(N-n+1)+1$, then $G$ contains a copy of $\AP_n$.

Similarly to the proof of Theorem~\ref{th:ramsey}, note that for any $i\in[\floor{n/2}]$ and $i+1\leq j\leq n-i$, the edge $ij$ cannot be a part of any copy of $\AP_n$ in $K_N^<$.
Also, for any $i \in [\ceil{n/2}-1]$ and $i+1\leq j\leq n-1-i$, the edge between $N+1-j$ and $N+1-i$ cannot be a part of any copy of $\AP_n$ in $K_N^<$.
We will refer to all the edges described in this paragraph as \emph{grey}.
As the edges mentioned above are pairwise distinct, the total number of grey edges is

$$\sum_{i=1}^{\lfloor n/2\rfloor}(n-2i)+\sum_{i=1}^{\lceil n/2\rceil-1}(n-2i-1)=\binom{n-1}{2}.$$

\input{thm3}

Partition the edges of $G$ as $E(G)=E\cup F$, where $F$ is the set of edges in $G$ that are grey and $E$ is the set of non-grey edges in $G$.
By the above, we have $|F|\leq \binom{n-1}{2}$.

We now describe an algorithm which removes edges of $E$ in $n-2$ steps in the following way.
In every odd step~$2i-1$ with $i\in[1,\ceil{n/2}-1]$, we remove the leftmost edge in $E$ adjacent to each vertex in $I_i=[n-i+1,N-i+1]$ (if they exist).
In every even step~$2i$ with $i\in[1,\floor{n/2}-1]$, we remove the rightmost edge in $E$ adjacent to each vertex in $J_i=[i+1,N-n+i+1]$ (if they exist).

Straightforward adaptations of the proofs of Claims \ref{claim:algorithm} and~\ref{claim:no_edges} show that if an edge in~$E$ remains after step $n-2$ of the algorithm, then $G$ contains an $\AP_n$.
In each step, the algorithm removes at most $N-n+1$ edges, and hence at most $(n-2)(N-n+1)$ of the edges in $E$ get removed during the $n-2$ steps of the algorithm. 
Note that 
\begin{align*}
    |E| = e(G) - |F|  &\ge 1+\binom{n-1}2+(n-2)(N-n+1) - \binom{n-1}2 \\
    &= 1+(n-2)(N-n+1).
\end{align*}
Therefore, at least one edge in $E$ remains after step $n-2$ of the algorithm, which implies that $G$ contains an $\AP_n$.

To finish the proof of the theorem, we give a construction of an ordered graph $G$ with vertex set $[N]$ with $\binom{n-1}2+(n-2)(N-n+1)$ edges that contains no $\AP_n$.
The edge set of $G$ consists exactly of all the edges adjacent to any vertex in $X=[1,\ceil{n/2}-1]$ or $Y=[N-\floor{n/2}+2,N]$.
It is straightforward to check that $G$ contains $\binom{n-1}2+(n-2)(N-n+1)$ edges.
If $G$ contains a copy of $\AP_n$, then as $|X|=\ceil{n/2}-1$ and $|Y|=\floor{n/2}-1$, the $(\ceil{n/2})$-th and $(\ceil{n/2}+1)$-th vertices in this copy of $\AP_n$ are both not in $X\cup Y$.
Hence, they are not connected by an edge in $G$, even though they must form an edge by definition of $\AP_n$, a contradiction.
\end{proof}

\noindent \textbf{Remark.} We remark once more that Theorem~\ref{th:turan} was independently obtained by Barát, Freschi, and Tóth~{\cite{barat2025matchings}}.
Interestingly, they provided a different lower bound construction by letting the edge set of an extremal graph consist of all edges $ij$ with $|i-j|\leq n-2$. 

\section{Discussion of other ordered paths}\label{sec4}

For $n$ even, recall the definitions of the ordered paths $P_n^{<,<}, P_n^{>,>}, P_n^{<,>}, P_n^{>,<}$ in Section~\ref{sec1.3}. 
We have already observed that $P^{>,<}_n$ is isomorphic to $P^{<,>}_n$ after reversing the order of the vertices, so they have the same ordered Ramsey and Tur\'an numbers.
Now, we consider the ordered Ramsey and Tur\'an numbers of $P_n^{>,>}$ and $P_n^{<,<}$.

First, we prove Proposition~\ref{prop:ramsey}, which gives an upper bound on the ordered Ramsey numbers of $P_n^{<,<}$ and $P_n^{>,>}$. 
The method is similar to those of Theorems \ref{th:ramsey} and \ref{th:turan}.

\begin{proof}[Proof of Proposition~\ref{prop:ramsey}]
Let $k = n/2$, and let $N=2M$ be an even number.
Let $A=[1,M]$ and $B=[M+1,N]$. 
Consider any red/blue edge-colouring of $K_N^<$.

Since $A$ and $B$ do not intersect, we have $e(A,B)=M^2$.
We run an algorithm that searches for a monochromatic copy of $P=(v_i)_{i=1}^n$ such that $(v_{2i-1})_{i=1}^k$ all lie in $A$ and $(v_{2i})_{i=1}^k$ all lie in $B$.
We define the \emph{grey edges} as edges between $A$ and $B$ that cannot be present in any copy of $P$ of the form described above. 

If $P$ is $P_n^{<,<}$, then the grey edges are
$$\{(i,N+1-j)\mid i \in [k-1], j\in[k-i]\}\cup\{(M+1-i,M+j)\mid i \in [k-2],j \in[k-i-1]\}.$$
If $P$ is $P_n^{>,>}$, then the grey edges are
$$\{(i,N+1-j)\mid i \in [k-2], j\in[k-i-1]\}\cup\{(M+1-i,M+j)\mid i \in [k-1],j \in[k-i]\}.$$
In both cases, the number of grey edges is $f=(k-1)^2$. See Figure \ref{fig:grey} for an illustration.

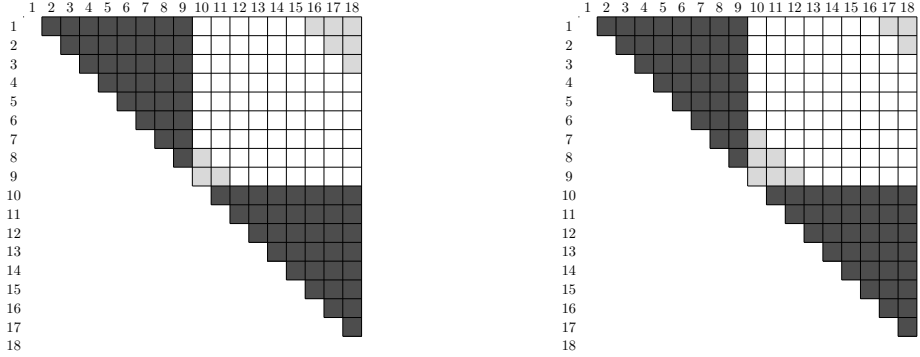
\begin{figure}[h]
\centering
\begin{subfigure}{0.48\textwidth}
\centering
\scalebox{0.5}{%
\begin{tikzpicture}

\fill[gray,opacity=0.3] (8.5,8.5) -- (8.5,9.0) -- (10.0,9.0) -- (10.0,8.5);;
\fill[gray,opacity=0.3] (9.0,8.5) -- (9.0,8.0) -- (10.0,8.0) -- (10.0,8.5);
\fill[gray,opacity=0.3] (9.5,7.5) -- (9.5,8.0) -- (10.0,8.0) -- (10.0,7.5);

\fill[gray,opacity=0.3] (5.5,5.5) -- (5.5,4.5) -- (6.0,4.5) -- (6.0,5.5);
\fill[gray,opacity=0.3] (6.0,4.5) -- (6.0,5.0) -- (6.5,5.0) -- (6.5,4.5);

\fill[black,opacity=0.7] (1.5,8.5) -- (1.5,9.0) -- (5.5,9.0) -- (5.5,8.5);
\fill[black,opacity=0.7] (2.0,8.5) -- (2.0,8.0) -- (5.5,8.0) -- (5.5,8.5);
\fill[black,opacity=0.7] (2.5,7.5) -- (2.5,8.0) -- (5.5,8.0) -- (5.5,7.5);
\fill[black,opacity=0.7] (3.0,7.5) -- (3.0,7.0) -- (5.5,7.0) -- (5.5,7.5);
\fill[black,opacity=0.7] (3.5,6.5) -- (3.5,7.0) -- (5.5,7.0) -- (5.5,6.5);
\fill[black,opacity=0.7] (4.0,6.5) -- (4.0,6.0) -- (5.5,6.0) -- (5.5,6.5);
\fill[black,opacity=0.7] (4.5,5.5) -- (4.5,6.0) -- (5.5,6.0) -- (5.5,5.5);
\fill[black,opacity=0.7] (5.0,5.5) -- (5.0,5.0) -- (5.5,5.0) -- (5.5,5.5);

\fill[black,opacity=0.7] (10.0,4.5) -- (10.0,0.5) -- (9.5,0.5) -- (9.5,4.5);
\fill[black,opacity=0.7] (9.0,4.5) -- (9.0,1.0) -- (9.5,1.0) -- (9.5,4.5);
\fill[black,opacity=0.7] (9.0,4.5) -- (9.0,1.5) -- (8.5,1.5) -- (8.5,4.5);
\fill[black,opacity=0.7] (8.0,4.5) -- (8.0,2.0) -- (8.5,2.0) -- (8.5,4.5);
\fill[black,opacity=0.7] (8.0,4.5) -- (8.0,2.5) -- (7.5,2.5) -- (7.5,4.5);
\fill[black,opacity=0.7] (7.0,4.5) -- (7.0,3.0) -- (7.5,3.0) -- (7.5,4.5);
\fill[black,opacity=0.7] (7.0,4.5) -- (7.0,3.5) -- (6.5,3.5) -- (6.5,4.5);
\fill[black,opacity=0.7] (6.0,4.5) -- (6.0,4.0) -- (6.5,4.0) -- (6.5,4.5);

\node[align=center] at (1.25,9.25) {$1$};
\node[align=center] at (1.75,9.25) {$2$};
\node[align=center] at (2.25,9.25) {$3$};
\node[align=center] at (2.75,9.25) {$4$};
\node[align=center] at (3.25,9.25) {$5$};
\node[align=center] at (3.75,9.25) {$6$};
\node[align=center] at (4.25,9.25) {$7$};
\node[align=center] at (4.75,9.25) {$8$};
\node[align=center] at (5.25,9.25) {$9$};
\node[align=center] at (5.75,9.25) {$10$};
\node[align=center] at (6.25,9.25) {$11$};
\node[align=center] at (6.75,9.25) {$12$};
\node[align=center] at (7.25,9.25) {$13$};
\node[align=center] at (7.75,9.25) {$14$};
\node[align=center] at (8.25,9.25) {$15$};
\node[align=center] at (8.75,9.25) {$16$};
\node[align=center] at (9.25,9.25) {$17$};
\node[align=center] at (9.75,9.25) {$18$};

\node[align=center] at (0.75,8.75) {$1$};
\node[align=center] at (0.75,8.25) {$2$};
\node[align=center] at (0.75,7.75) {$3$};
\node[align=center] at (0.75,7.25) {$4$};
\node[align=center] at (0.75,6.75) {$5$};
\node[align=center] at (0.75,6.25) {$6$};
\node[align=center] at (0.75,5.75) {$7$};
\node[align=center] at (0.75,5.25) {$8$};
\node[align=center] at (0.75,4.75) {$9$};
\node[align=center] at (0.75,4.25) {$10$};
\node[align=center] at (0.75,3.75) {$11$};
\node[align=center] at (0.75,3.25) {$12$};
\node[align=center] at (0.75,2.75) {$13$};
\node[align=center] at (0.75,2.25) {$14$};
\node[align=center] at (0.75,1.75) {$15$};
\node[align=center] at (0.75,1.25) {$16$};
\node[align=center] at (0.75,0.75) {$17$};
\node[align=center] at (0.75,0.25) {$18$};

\draw (9.5,0.5) -- (10.0,0.5);
\draw (9.0,1.0) -- (10.0,1.0);
\draw (8.5,1.5) -- (10.0,1.5);
\draw (8.0,2.0) -- (10.0,2.0);
\draw (7.5,2.5) -- (10.0,2.5);
\draw (7.0,3.0) -- (10.0,3.0);
\draw (6.5,3.5) -- (10.0,3.5);
\draw (6.0,4.0) -- (10.0,4.0);
\draw (5.5,4.5) -- (10.0,4.5);
\draw (5.0,5.0) -- (10.0,5.0);
\draw (4.5,5.5) -- (10.0,5.5);
\draw (4.0,6.0) -- (10.0,6.0);
\draw (3.5,6.5) -- (10.0,6.5);
\draw (3.0,7.0) -- (10.0,7.0);
\draw (2.5,7.5) -- (10.0,7.5);
\draw (2.0,8.0) -- (10.0,8.0);
\draw (1.5,8.5) -- (10.0,8.5);
\draw (1.5,9.0) -- (10.0,9.0);

\draw (10.0,0.5) -- (10.0,9.0);
\draw (9.5,0.5) -- (9.5,9.0);
\draw (9.0,1.0) -- (9.0,9.0);
\draw (8.5,1.5) -- (8.5,9.0);
\draw (8.0,2.0) -- (8.0,9.0);
\draw (7.5,2.5) -- (7.5,9.0);
\draw (7.0,3.0) -- (7.0,9.0);
\draw (6.5,3.5) -- (6.5,9.0);
\draw (6.0,4.0) -- (6.0,9.0);
\draw (5.5,4.5) -- (5.5,9.0);
\draw (5.0,5.0) -- (5.0,9.0);
\draw (4.5,5.5) -- (4.5,9.0);
\draw (4.0,6.0) -- (4.0,9.0);
\draw (3.5,6.5) -- (3.5,9.0);
\draw (3.0,7.0) -- (3.0,9.0);
\draw (2.5,7.5) -- (2.5,9.0);
\draw (2.0,8.0) -- (2.0,9.0);
\draw (1.5,8.5) -- (1.5,9.0);
\draw (1.0,9.0) -- (1.0,9.0);
\end{tikzpicture}
}
\end{subfigure}
\begin{subfigure}{0.48\textwidth}
\centering
\scalebox{0.5}{%
\begin{tikzpicture}

\fill[gray,opacity=0.3] (9.0,8.5) -- (9.0,9.0) -- (10.0,9.0) -- (10.0,8.5);
\fill[gray,opacity=0.3] (9.5,8.5) -- (9.5,8.0) -- (10.0,8.0) -- (10.0,8.5);

\fill[gray,opacity=0.3] (5.5,6.0) -- (5.5,4.5) -- (6.0,4.5) -- (6.0,6.0);
\fill[gray,opacity=0.3] (6.0,4.5) -- (6.0,5.5) -- (6.5,5.5) -- (6.5,4.5);
\fill[gray,opacity=0.3] (7.0,4.5) -- (7.0,5.0) -- (6.5,5.0) -- (6.5,4.5);

\fill[black,opacity=0.7] (1.5,8.5) -- (1.5,9.0) -- (5.5,9.0) -- (5.5,8.5);
\fill[black,opacity=0.7] (2.0,8.5) -- (2.0,8.0) -- (5.5,8.0) -- (5.5,8.5);
\fill[black,opacity=0.7] (2.5,7.5) -- (2.5,8.0) -- (5.5,8.0) -- (5.5,7.5);
\fill[black,opacity=0.7] (3.0,7.5) -- (3.0,7.0) -- (5.5,7.0) -- (5.5,7.5);
\fill[black,opacity=0.7] (3.5,6.5) -- (3.5,7.0) -- (5.5,7.0) -- (5.5,6.5);
\fill[black,opacity=0.7] (4.0,6.5) -- (4.0,6.0) -- (5.5,6.0) -- (5.5,6.5);
\fill[black,opacity=0.7] (4.5,5.5) -- (4.5,6.0) -- (5.5,6.0) -- (5.5,5.5);
\fill[black,opacity=0.7] (5.0,5.5) -- (5.0,5.0) -- (5.5,5.0) -- (5.5,5.5);

\fill[black,opacity=0.7] (10.0,4.5) -- (10.0,0.5) -- (9.5,0.5) -- (9.5,4.5);
\fill[black,opacity=0.7] (9.0,4.5) -- (9.0,1.0) -- (9.5,1.0) -- (9.5,4.5);
\fill[black,opacity=0.7] (9.0,4.5) -- (9.0,1.5) -- (8.5,1.5) -- (8.5,4.5);
\fill[black,opacity=0.7] (8.0,4.5) -- (8.0,2.0) -- (8.5,2.0) -- (8.5,4.5);
\fill[black,opacity=0.7] (8.0,4.5) -- (8.0,2.5) -- (7.5,2.5) -- (7.5,4.5);
\fill[black,opacity=0.7] (7.0,4.5) -- (7.0,3.0) -- (7.5,3.0) -- (7.5,4.5);
\fill[black,opacity=0.7] (7.0,4.5) -- (7.0,3.5) -- (6.5,3.5) -- (6.5,4.5);
\fill[black,opacity=0.7] (6.0,4.5) -- (6.0,4.0) -- (6.5,4.0) -- (6.5,4.5);

\node[align=center] at (1.25,9.25) {$1$};
\node[align=center] at (1.75,9.25) {$2$};
\node[align=center] at (2.25,9.25) {$3$};
\node[align=center] at (2.75,9.25) {$4$};
\node[align=center] at (3.25,9.25) {$5$};
\node[align=center] at (3.75,9.25) {$6$};
\node[align=center] at (4.25,9.25) {$7$};
\node[align=center] at (4.75,9.25) {$8$};
\node[align=center] at (5.25,9.25) {$9$};
\node[align=center] at (5.75,9.25) {$10$};
\node[align=center] at (6.25,9.25) {$11$};
\node[align=center] at (6.75,9.25) {$12$};
\node[align=center] at (7.25,9.25) {$13$};
\node[align=center] at (7.75,9.25) {$14$};
\node[align=center] at (8.25,9.25) {$15$};
\node[align=center] at (8.75,9.25) {$16$};
\node[align=center] at (9.25,9.25) {$17$};
\node[align=center] at (9.75,9.25) {$18$};

\node[align=center] at (0.75,8.75) {$1$};
\node[align=center] at (0.75,8.25) {$2$};
\node[align=center] at (0.75,7.75) {$3$};
\node[align=center] at (0.75,7.25) {$4$};
\node[align=center] at (0.75,6.75) {$5$};
\node[align=center] at (0.75,6.25) {$6$};
\node[align=center] at (0.75,5.75) {$7$};
\node[align=center] at (0.75,5.25) {$8$};
\node[align=center] at (0.75,4.75) {$9$};
\node[align=center] at (0.75,4.25) {$10$};
\node[align=center] at (0.75,3.75) {$11$};
\node[align=center] at (0.75,3.25) {$12$};
\node[align=center] at (0.75,2.75) {$13$};
\node[align=center] at (0.75,2.25) {$14$};
\node[align=center] at (0.75,1.75) {$15$};
\node[align=center] at (0.75,1.25) {$16$};
\node[align=center] at (0.75,0.75) {$17$};
\node[align=center] at (0.75,0.25) {$18$};

\draw (9.5,0.5) -- (10.0,0.5);
\draw (9.0,1.0) -- (10.0,1.0);
\draw (8.5,1.5) -- (10.0,1.5);
\draw (8.0,2.0) -- (10.0,2.0);
\draw (7.5,2.5) -- (10.0,2.5);
\draw (7.0,3.0) -- (10.0,3.0);
\draw (6.5,3.5) -- (10.0,3.5);
\draw (6.0,4.0) -- (10.0,4.0);
\draw (5.5,4.5) -- (10.0,4.5);
\draw (5.0,5.0) -- (10.0,5.0);
\draw (4.5,5.5) -- (10.0,5.5);
\draw (4.0,6.0) -- (10.0,6.0);
\draw (3.5,6.5) -- (10.0,6.5);
\draw (3.0,7.0) -- (10.0,7.0);
\draw (2.5,7.5) -- (10.0,7.5);
\draw (2.0,8.0) -- (10.0,8.0);
\draw (1.5,8.5) -- (10.0,8.5);
\draw (1.5,9.0) -- (10.0,9.0);

\draw (10.0,0.5) -- (10.0,9.0);
\draw (9.5,0.5) -- (9.5,9.0);
\draw (9.0,1.0) -- (9.0,9.0);
\draw (8.5,1.5) -- (8.5,9.0);
\draw (8.0,2.0) -- (8.0,9.0);
\draw (7.5,2.5) -- (7.5,9.0);
\draw (7.0,3.0) -- (7.0,9.0);
\draw (6.5,3.5) -- (6.5,9.0);
\draw (6.0,4.0) -- (6.0,9.0);
\draw (5.5,4.5) -- (5.5,9.0);
\draw (5.0,5.0) -- (5.0,9.0);
\draw (4.5,5.5) -- (4.5,9.0);
\draw (4.0,6.0) -- (4.0,9.0);
\draw (3.5,6.5) -- (3.5,9.0);
\draw (3.0,7.0) -- (3.0,9.0);
\draw (2.5,7.5) -- (2.5,9.0);
\draw (2.0,8.0) -- (2.0,9.0);
\draw (1.5,8.5) -- (1.5,9.0);
\draw (1.0,9.0) -- (1.0,9.0);
\end{tikzpicture}
}
\end{subfigure}
\caption{An illustration of the grey edges from the proof of Proposition \ref{prop:ramsey} for $N=18$ and $n=8$. On the left is the case of $P_n^{<,<}$, on the right the case of $P_n^{>,>}$. The dark squares correspond to edges that are not between $A$ and $B$, while the grey ones correspond to the grey edges in the proof.}\label{fig:grey}
\end{figure}

If $P$ is $P_n^{<,<}$, we define $I_i=[M+i,N-k+i]$ and $J_i=[i+1,M-k+i+1]$ for every $i \in [k-1]$.
The algorithm then runs for $n-2$ steps.
For every $i\in[k-1]$, in odd step~$2i-1$, the leftmost red and blue edges adjacent to each vertex in $I_i$ are removed, and in even step $2i$, the leftmost red and blue edges adjacent to each vertex in $J_i$ are removed.

On the other hand, if $P$ is $P_n^{>,>}$, for every $i \in [k-1]$, we define $I_i=[M+k-i+1,N-i+1]$ and $J_i=[k-i,M-i]$.
For every odd step $2i-1$, the rightmost red and blue edges adjacent to each vertex in $I_i$ are removed.
For every even step $2i$, the rightmost red and blue edges adjacent to each vertex in $J_i$ are removed.

In either case, one can show, similarly to the proofs of Theorems~\ref{th:ramsey} and~\ref{th:turan}, that if any non-grey edge between $A$ and $B$ remains after $n-2$ steps, then there is a monochromatic copy of $P$ in the red/blue edge-colouring of $K_N^<$.
Note that the number of removed edges in both cases is at most $2(n-2)(M-k+1)$.
Hence, a monochromatic copy of $P$ is present if $e(A,B) > 2(n-2)(M-k+1)+f$. This inequality is equivalent to $M^2>2(2k-2)(M-k+1)+(k-1)^2$, and holds if and only if $(M-3(k-1))(M-(k-1))>0$. 

In particular, if $M\geq3k-2$, then there is a monochromatic copy of $P$ in the red/blue edge-coloured $K_N^<$.
Therefore, $R_<(P)\leq N=2M=6k-4=3n-4$.
\end{proof}

We remark that a naive attempt to attain a better upper bound for $R_<(P)$ when $P\in\{P^{<,<}_n,P^{>,>}_n\}$ by allowing $A$ and $B$ to intersect as in the proof of Theorem~\ref{th:ramsey} runs into the following issue:
After the corresponding edge-deletion algorithm, the analogue of
Claim~\ref{claim:no_edges} does not need to hold, and we outline why below.
When $P=\AP_n$ and an edge~$v_{n-1}v_n$ with $v_{n-1} \in A$ and $v_n \in B$ remains after the deletion algorithm, we backtrack the algorithm and get a sequence of vertices $v_1, v_2, \dots, v_{n}$ such that $v_1, v_3, \ldots, v_{n-1} \in A$, $v_2, v_4, \ldots, v_n \in B$, and $v_iv_{i+1}$ is always an edge.
Moreover it is always the case that $v_1 < v_3 < \cdots < v_{n-1} < v_n < v_{n-2} < \cdots < v_2$, so $v_1v_2 \dots v_n$ is a copy of $\AP_n$.
Suppose we do the analogue for $P= P^{<,<}_n$, starting again with an edge $v_{n-1}v_n$ with $v_{n-1} \in A$ and $v_n \in B$.
We can still backtrack and get a sequence of vertices $v_1, v_2, \dots, v_{n}$ such that $v_1,v_3,\ldots, v_{n-1} \in A$, $v_2,v_4,\ldots,v_n \in B$, and $v_i v_{i+1}$ is always an edge.
However, while we have $v_1 < v_3 <\cdots< v_{n-1}$ and $v_2 < v_4 <\cdots <v_n$, it is not always the case that
\begin{equation}\label{eq:maxmin}
v_{n-1} < v_2,
\end{equation}
which is required for $v_1v_2 \dots v_n$ to be a valid copy of $P_n^{<,<}$. In fact, it may even happen that the same vertex appears twice in the sequence.
Therefore, we do not necessarily get a copy of $P$.
From the point of view of the matrix representation, this is due to the fact that the staircase corresponding to a copy of $P^{<,<}_n$ has a different ``orientation'' than that corresponding to a copy of $\AP_n$.
Note that if $A$ and $B$ are disjoint as in the proof of Proposition~\ref{prop:ramsey}, then condition~\eqref{eq:maxmin} is automatically satisfied and thus, in this setup, the analogue of Claim~\ref{claim:no_edges} does hold.

It turns out that the ordered Turán numbers of $P_n^{<,<}$ and $P_n^{>,>}$ can differ both with each other and with that of $P_n^{<,>}$ and $P_n^{>,<}$.
For example, we have the following values and Figure~\ref{fig:turanother} shows the respective extremal constructions.

\[\arraycolsep=12pt\def\arraystretch{1.25}
\begin{array}{ccc} 
\mathrm{ex}_<(6,P_4^{<,>})=\mathrm{ex}_<(6,P_4^{>,<})=9\phantom{0} & \mathrm{ex}_<(6,P_4^{<,<})=9\phantom{0} & \mathrm{ex}_<(6,P_4^{>,>})=11\\
\mathrm{ex}_<(8,P_6^{<,>})=\mathrm{ex}_<(8,P_6^{>,<})=22 & \mathrm{ex}_<(8,P_6^{<,<})=24 & \mathrm{ex}_<(8,P_6^{>,>})=24
\end{array}\]

\begin{figure}[h]
\centering
\begin{subfigure}{0.48\textwidth}
\centering
\scalebox{1.0}{%
\begin{tikzpicture}

\fill[green,opacity=0.7] (1.5,3.0) -- (1.5,3.5) -- (4.0,3.5) -- (4.0,3.0);
\fill[green,opacity=0.7] (3.5,1.0) -- (3.5,3.0) -- (4.0,3.0) -- (4.0,1.0);

\fill[white] (3.5,0.5) -- (3.5,0.0) -- (3.0,0.0) -- (3.0,0.5);

\node[align=center] at (1.25,3.75) {$1$};
\node[align=center] at (1.75,3.75) {$2$};
\node[align=center] at (2.25,3.75) {$3$};
\node[align=center] at (2.75,3.75) {$4$};
\node[align=center] at (3.25,3.75) {$5$};
\node[align=center] at (3.75,3.75) {$6$};

\node[align=center] at (0.75,3.25) {$1$};
\node[align=center] at (0.75,2.75) {$2$};
\node[align=center] at (0.75,2.25) {$3$};
\node[align=center] at (0.75,1.75) {$4$};
\node[align=center] at (0.75,1.25) {$5$};
\node[align=center] at (0.75,0.75) {$6$};


\draw (3.5,1.0) -- (4.0,1.0);
\draw (3.0,1.5) -- (4.0,1.5);
\draw (2.5,2.0) -- (4.0,2.0);
\draw (2.0,2.5) -- (4.0,2.5);
\draw (1.5,3.0) -- (4.0,3.0);
\draw (1.5,3.5) -- (4.0,3.5);


\draw (4.0,1.0) -- (4.0,3.5);
\draw (3.5,1.0) -- (3.5,3.5);
\draw (3.0,1.5) -- (3.0,3.5);
\draw (2.5,2.0) -- (2.5,3.5);
\draw (2.0,2.5) -- (2.0,3.5);
\draw (1.5,3.0) -- (1.5,3.5);
\end{tikzpicture}
}
\end{subfigure}
\begin{subfigure}{0.48\textwidth}
\centering

\scalebox{1.0}{%
\begin{tikzpicture}

\fill[green,opacity=0.7] (1.5,3.0) -- (1.5,3.5) -- (3.5,3.5) -- (3.5,3.0);
\fill[green,opacity=0.7] (3.5,1.0) -- (3.5,3.0) -- (4.0,3.0) -- (4.0,1.0);
\fill[green,opacity=0.7] (2.5,2.5) -- (2.5,3.0) -- (2.0,3.0) -- (2.0,2.5);
\fill[green,opacity=0.7] (2.5,2.5) -- (2.5,2.0) -- (3.0,2.0) -- (3.0,2.5);
\fill[green,opacity=0.7] (3.5,1.5) -- (3.5,2.0) -- (3.0,2.0) -- (3.0,1.5);

\fill[white] (3.5,0.5) -- (3.5,0.0) -- (3.0,0.0) -- (3.0,0.5);

\node[align=center] at (1.25,3.75) {$1$};
\node[align=center] at (1.75,3.75) {$2$};
\node[align=center] at (2.25,3.75) {$3$};
\node[align=center] at (2.75,3.75) {$4$};
\node[align=center] at (3.25,3.75) {$5$};
\node[align=center] at (3.75,3.75) {$6$};

\node[align=center] at (0.75,3.25) {$1$};
\node[align=center] at (0.75,2.75) {$2$};
\node[align=center] at (0.75,2.25) {$3$};
\node[align=center] at (0.75,1.75) {$4$};
\node[align=center] at (0.75,1.25) {$5$};
\node[align=center] at (0.75,0.75) {$6$};


\draw (3.5,1.0) -- (4.0,1.0);
\draw (3.0,1.5) -- (4.0,1.5);
\draw (2.5,2.0) -- (4.0,2.0);
\draw (2.0,2.5) -- (4.0,2.5);
\draw (1.5,3.0) -- (4.0,3.0);
\draw (1.5,3.5) -- (4.0,3.5);


\draw (4.0,1.0) -- (4.0,3.5);
\draw (3.5,1.0) -- (3.5,3.5);
\draw (3.0,1.5) -- (3.0,3.5);
\draw (2.5,2.0) -- (2.5,3.5);
\draw (2.0,2.5) -- (2.0,3.5);
\draw (1.5,3.0) -- (1.5,3.5);
\end{tikzpicture}
}
\end{subfigure}

\begin{subfigure}{0.48\textwidth}
\centering
\scalebox{1.0}{%
\begin{tikzpicture}

\fill[green,opacity=0.7] (1.5,3.0) -- (1.5,3.5) -- (5.0,3.5) -- (5.0,3.0);
\fill[green,opacity=0.7] (2.0,3.0) -- (2.0,2.5) -- (4.0,2.5) -- (4.0,3.0);
\fill[green,opacity=0.7] (5.0,3.0) -- (5.0,2.5) -- (4.5,2.5) -- (4.5,3.0);
\fill[green,opacity=0.7] (2.5,2.0) -- (3.0,2.0) -- (3.0,2.5) -- (2.5,2.5);
\fill[green,opacity=0.7] (5.0,2.0) -- (5.0,2.5) -- (4.0,2.5) -- (4.0,2.0);
\fill[green,opacity=0.7] (3.5,2.0) -- (3.0,2.0) -- (3.0,1.5) -- (3.5,1.5);
\fill[green,opacity=0.7] (5.0,2.0) -- (5.0,1.5) -- (4.0,1.5) -- (4.0,2.0);
\fill[green,opacity=0.7] (5.0,1.0) -- (5.0,1.5) -- (3.5,1.5) -- (3.5,1.0);
\fill[green,opacity=0.7] (5.0,1.0) -- (5.0,0.5) -- (4.0,0.5) -- (4.0,1.0);
\fill[green,opacity=0.7] (5.0,0.0) -- (5.0,0.5) -- (4.5,0.5) -- (4.5,0.0);

\node[align=center] at (1.25,3.75) {$1$};
\node[align=center] at (1.75,3.75) {$2$};
\node[align=center] at (2.25,3.75) {$3$};
\node[align=center] at (2.75,3.75) {$4$};
\node[align=center] at (3.25,3.75) {$5$};
\node[align=center] at (3.75,3.75) {$6$};
\node[align=center] at (4.25,3.75) {$7$};
\node[align=center] at (4.75,3.75) {$8$};

\node[align=center] at (0.75,3.25) {$1$};
\node[align=center] at (0.75,2.75) {$2$};
\node[align=center] at (0.75,2.25) {$3$};
\node[align=center] at (0.75,1.75) {$4$};
\node[align=center] at (0.75,1.25) {$5$};
\node[align=center] at (0.75,0.75) {$6$};
\node[align=center] at (0.75,0.25) {$7$};
\node[align=center] at (0.75,-0.25) {$8$};


\draw (4.5,0.0) -- (5.0,0.0);
\draw (4.0,0.5) -- (5.0,0.5);
\draw (3.5,1.0) -- (5.0,1.0);
\draw (3.0,1.5) -- (5.0,1.5);
\draw (2.5,2.0) -- (5.0,2.0);
\draw (2.0,2.5) -- (5.0,2.5);
\draw (1.5,3.0) -- (5.0,3.0);
\draw (1.5,3.5) -- (5.0,3.5);


\draw (5.0,0.0) -- (5.0,3.5);
\draw (4.5,0.0) -- (4.5,3.5);
\draw (4.0,0.5) -- (4.0,3.5);
\draw (3.5,1.0) -- (3.5,3.5);
\draw (3.0,1.5) -- (3.0,3.5);
\draw (2.5,2.0) -- (2.5,3.5);
\draw (2.0,2.5) -- (2.0,3.5);
\draw (1.5,3.0) -- (1.5,3.5);
\end{tikzpicture}
}
\end{subfigure}
\begin{subfigure}{0.48\textwidth}
\centering
\scalebox{1.0}{%
\begin{tikzpicture}

\fill[green,opacity=0.7] (1.5,3.0) -- (1.5,3.5) -- (4.0,3.5) -- (4.0,3.0);
\fill[green,opacity=0.7] (5.0,3.0) -- (5.0,3.5) -- (4.5,3.5) -- (4.5,3.0);
\fill[green,opacity=0.7] (2.0,3.0) -- (2.0,2.5) -- (4.0,2.5) -- (4.0,3.0);
\fill[green,opacity=0.7] (2.5,2.0) -- (3.5,2.0) -- (3.5,2.5) -- (2.5,2.5);
\fill[green,opacity=0.7] (5.0,2.0) -- (5.0,2.5) -- (4.0,2.5) -- (4.0,2.0);
\fill[green,opacity=0.7] (5.0,2.0) -- (3.0,2.0) -- (3.0,1.5) -- (5.0,1.5);
\fill[green,opacity=0.7] (5.0,1.0) -- (5.0,1.5) -- (3.5,1.5) -- (3.5,1.0);
\fill[green,opacity=0.7] (5.0,1.0) -- (5.0,0.5) -- (4.0,0.5) -- (4.0,1.0);
\fill[green,opacity=0.7] (5.0,0.0) -- (5.0,0.5) -- (4.5,0.5) -- (4.5,0.0);

\node[align=center] at (1.25,3.75) {$1$};
\node[align=center] at (1.75,3.75) {$2$};
\node[align=center] at (2.25,3.75) {$3$};
\node[align=center] at (2.75,3.75) {$4$};
\node[align=center] at (3.25,3.75) {$5$};
\node[align=center] at (3.75,3.75) {$6$};
\node[align=center] at (4.25,3.75) {$7$};
\node[align=center] at (4.75,3.75) {$8$};

\node[align=center] at (0.75,3.25) {$1$};
\node[align=center] at (0.75,2.75) {$2$};
\node[align=center] at (0.75,2.25) {$3$};
\node[align=center] at (0.75,1.75) {$4$};
\node[align=center] at (0.75,1.25) {$5$};
\node[align=center] at (0.75,0.75) {$6$};
\node[align=center] at (0.75,0.25) {$7$};
\node[align=center] at (0.75,-0.25) {$8$};


\draw (4.5,0.0) -- (5.0,0.0);
\draw (4.0,0.5) -- (5.0,0.5);
\draw (3.5,1.0) -- (5.0,1.0);
\draw (3.0,1.5) -- (5.0,1.5);
\draw (2.5,2.0) -- (5.0,2.0);
\draw (2.0,2.5) -- (5.0,2.5);
\draw (1.5,3.0) -- (5.0,3.0);
\draw (1.5,3.5) -- (5.0,3.5);


\draw (5.0,0.0) -- (5.0,3.5);
\draw (4.5,0.0) -- (4.5,3.5);
\draw (4.0,0.5) -- (4.0,3.5);
\draw (3.5,1.0) -- (3.5,3.5);
\draw (3.0,1.5) -- (3.0,3.5);
\draw (2.5,2.0) -- (2.5,3.5);
\draw (2.0,2.5) -- (2.0,3.5);
\draw (1.5,3.0) -- (1.5,3.5);
\end{tikzpicture}
}
\end{subfigure}
\caption{Extremal constructions showing $\mathrm{ex}_<(6,P_4^{<,<})=9$ (top left), $\mathrm{ex}_<(6,P_4^{>,>})\\=11$ (top right), $\mathrm{ex}_<(8,P_6^{<,<})=24$ (bottom left), and $\mathrm{ex}_<(8,P_6^{>,>})=24$ (bottom right).}\label{fig:turanother}
\end{figure}
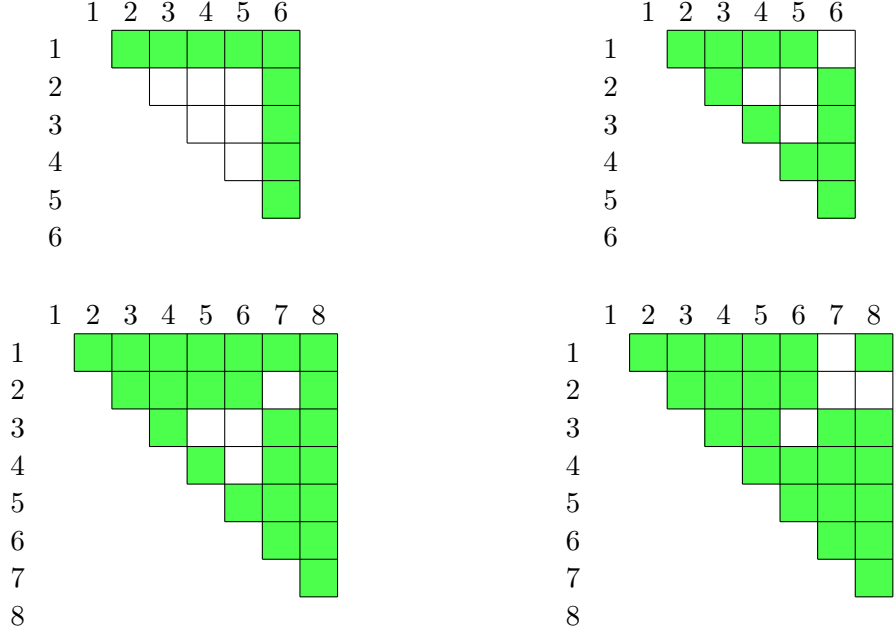

For $P \in \{P^{<,<}_n,P^{>,>}_n\}$, our proof of the upper bound on $\ex_<(N,P)$ in Theorem~\ref{thm:turanother} uses the \emph{ordered bipartite Tur\'an number} $\ex_<(M,M,P)$ of $P$, which is the maximum number of edges an ordered bipartite graph with bipartition classes $[1,M]$ and $[M+1,2M]$ can have if it avoids $P$.
In contrast to the behaviours of the ordered Tur\'an numbers above, the four ordered paths considered in this section share the same ordered bipartite Tur\'an numbers, which are determined exactly by the following result.

\begin{theorem}\label{thm:bipturan}
Let $N$ and $n$ be even with $N\geq n$. Let $P \in \{P_n^{<,<}, P_n^{>,>}, P_n^{<,>},$ $P_n^{>,<}\}$. Then, $\mathrm{ex}_<(N/2,N/2,P)= (n/2-1)(N-n/2+1)$.
\end{theorem}

\begin{proof}
The proof is similar to the proofs of Theorem \ref{th:ramsey} and Proposition \ref{prop:ramsey}.
Let $M=N/2$ and $k=n/2$.

Let $A=[1,M]$ and $B=[M+1,N]$, and let $G$ be an ordered bipartite graph with bipartition classes $A$ and $B$.
First, we show that if $e(G)>(n/2-1)(N-n/2+1)$, then $G$ contains a copy of $P=(v_i)_{i=1}^n$.
Note that in such a copy, $(v_{2i-1})_{i=1}^k$ all lie in $A$ and $(v_{2i})_{i=1}^k$ all lie in $B$. 

Recall from the proofs of Theorem~\ref{th:ramsey} and Proposition \ref{prop:ramsey} that there are $(k-1)^2$ grey edges in $G$ that are never present in any copy of $P$.
We then run the same algorithm as in the proofs of Theorem~\ref{th:ramsey} and Proposition \ref{prop:ramsey} with the difference that instead of removing the leftmost (or rightmost) red and blue edges in each step, we remove the leftmost (or rightmost) edge.
In each of the $n-2$ steps, at most $M-k+1$ edges are removed. Once again, if any non-grey edge remains after $n-2$ of the algorithm, then there is a copy of $P$ in $G$.
Therefore, if $e(G)>(n-2)(M-k+1)+(k-1)^2=(n/2-1)(N-n/2+1)$, then $G$ contains a copy of $P$. 

On the other hand, if we let $G$ be the ordered bipartite graph with bipartition classes~$A$ and $B$, and edge set
\begin{align*}
&\{ij:i\in[k-1],j\in B\}\cup\{ij:i\in A,j\in[M+1,M+k-1]\} & \text{if }P=P_n^{<,<},\\
&\{ij:i\in[M-k+2,M],j\in B\}\cup\{ij:i\in A,j\in[N-k+2,N]\} & \text{if }P=P_n^{>,>},\\
&\{ij:i\in[k-1],j\in B\}\cup\{ij:i\in A,j\in[N-k+2,N]\} & \text{if }P=P_n^{<,>},\\
&\{ij:i\in[M-k+2,M],j\in B\}\cup\{ij:i\in A,j\in[M+1,M+k-1]\} & \text{if }P=P_n^{>,<}.
\end{align*}
Then, $e(G)=2(k-1)M-(k-1)^2=(n/2-1)(N-n/2+1)$ in all cases. 
We show that $G$ contains no copy of $P=P_n^{<,<}$ (the other cases are similar).
Indeed, if $G$ contains a copy of $P$, then the $k$-th vertex in $P$ must lie in $A\setminus[k-1]$ and the $(2k)$-th vertex in $P$ must lie in $B\setminus[M+1,M+k-1]$, which means they do not form an edge in $G$ by definition, a contradiction. 
\end{proof}

To prove Theorem~\ref{thm:turanother}, which gives an upper bound on the ordered Tur\'an numbers of~$P_n^{<,<}$ and $P_n^{>,>}$, we iteratively partition the current vertex set in half, and apply Theorem~\ref{thm:bipturan}, our bipartite Tur\'an result. 
\begin{proof}[Proof of Theorem~\ref{thm:turanother}]
First, consider the case when $N/n=2^{t-1}$ for some integer $t\geq1$. We use induction on $t$ to show that $\mathrm{ex}_<(N,P)\leq \frac12nN(\log_2(N/n)+1)$. When $t=1$, this is immediate as $N=n$, and $\mathrm{ex}_<(N,P)\leq\binom{N}2\leq\frac12N^2=\frac12nN(\log_2(N/n)+1)$.
Now suppose this holds for some $t\geq 1$, and consider the case when $N/n=2^t$.
Let $A=[1,N/2]$ and $B=[N/2+1,N]$.
If $G$ is an ordered graph on $N$ vertices containing no copy of $P$, then $e(A),e(B)\leq\ex_<(N/2,P)$ and $e(A,B)\leq\ex_<(N/2,N/2,P)$.
Hence, by induction hypothesis and Theorem~\ref{thm:bipturan},
\begin{align*}
\ex_<(N,P)&\leq2\ex_<(N/2,P)+\ex_<(N/2,N/2,P)\\
&\leq\frac12nN(\log_2(N/2n)+1)+\left(\frac12n-1\right)\left(N-\frac12n+1\right)\\
&\leq\frac12nN(\log_2(N/2n)+2)=\frac12nN(\log_2(N/n)+1),
\end{align*}
which completes the induction. 

Now for the general case, let $t=1+\floor{\log_2(N/n)}$, so $2^{t-1}n\leq N<2^tn$.
Then, from above, $\ex_<(N,P)\leq\ex_<(2^tn,P)\leq2^{t-1}n^2(t+1)\leq nN(\log_2(N/n)+2)$.
\end{proof}

\section*{Acknowledgments}

This project originated in the online workshop \emph{Topics in Ramsey theory} organised by the Sparse Graphs Coalition.
We are grateful to Stijn Cambie, Nemanja Draganić, António Gir\~{a}o, Eoin Hurley, and Ross Kang for organising the workshop, and to Kalina Petrova for proposing this problem.
We also thank Andrea Freschi for letting us know about the related independent work in~\cite{barat2025matchings} while it was still in preparation, and Marian Poljak for sharing computational data.

\bibliographystyle{abbrvnat}  
\bibliography{bibliography}

\end{document}